\documentclass{amsart}

\usepackage{amsmath,amssymb,amsthm,a4wide}
\usepackage{graphicx,tikz}

\newtheorem*{thm}{Theorem}

\newtheorem{lemma}{Lemma}

\theoremstyle{definition}

\theoremstyle{remark}

\begin{document}

\title[]{ Quantitative Homogenization and\\ convergence of moving averages }
\thanks{This work is supported by the NSF (DMS-1763179) and the Alfred P. Sloan Foundation. Part of this work was
done at PCMI, the author is grateful for its hospitality.}
\author[]{Stefan Steinerberger}
\address{Department of Mathematics, Yale University, New Haven, CT 06511, USA}
\email{stefan.steinerberger@yale.edu}

\begin{abstract} We study homogenization it its most basic form
$$-\left(a\left(\frac{x}{\varepsilon}\right) u_{\varepsilon}'(x)\right)' = f(x) \quad \mbox{for} ~x \in (0,1),$$
where $a(\cdot)$ is a positive $1-$periodic continuous function, $f$ is smooth and $u_{\varepsilon}$ is subjected to Dirichlet boundary conditions.
Classically, there is a homogenized equation with $a(\cdot)$ replaced by a constant coefficient $\overline{a} > 0$
whose solution $u$ satisfies $\|u-u_{\varepsilon}\|_{L^{\infty}} \lesssim \varepsilon$. We show that local averages can result in faster convergence:
for example, if $a(x) = a(1-x)$, then for $x \in (\varepsilon, 1-\varepsilon)$
$$ \left| \frac{1}{\varepsilon} \int_{x-\varepsilon/2}^{x+\varepsilon/2}{ u_{\varepsilon}(y) dy} - u(x) \right| \lesssim_{a, f}  \varepsilon^2.$$
If the condition on $a(\cdot)$ is not satisfied, then subtracting an explicitly given linear function (depending on $a(\cdot),f$) results in the same bound. Moreover, for certain right-hand sides $f$ we have the result for all $a(\cdot)$, which seems like an interesting phenomenon.
 We also describe another approach to quantitative homogenization problems and illustrate it on the same example.
\end{abstract}

\maketitle

\section{Introduction and results}

\subsection{Introduction.} 
We study the basic ordinary differential equation
\begin{align*}
-\left(a\left(\frac{x}{\varepsilon}\right) u_{\varepsilon}'(x)\right)' &= f(x) \quad \mbox{for} ~x \in (0,1)\\
u_{\varepsilon}(0) = 0& = u_{\varepsilon}(1)
\end{align*}
where $a:\mathbb{R} \rightarrow \mathbb{R}_{+}$ is a strictly positive function with period 1, the Dirichlet
boundary conditions $u_{\varepsilon}(0) = 0 = u_{\varepsilon}(1)$ are prescribed, and $f$ is assumed to be smooth.
It is classical \cite{avi, ben, jikov, papa, pap} that $u_{\varepsilon}$ converges to a limiting function $u$ as $\varepsilon \rightarrow 0^+$, where
$u$ solves the homogenized equation
\begin{align*}
-\left(a u_{}'(x)\right)' &= f(x) \quad \mbox{for} ~x \in (0,1)\\
u_{}(0) = 0& = u_{}(1) 
\end{align*}
and the constant $a > 0$ is given by
$$ a = \left( \int_{0}^{1}{\frac{1}{a(x)} dx}\right)^{-1}.$$
Moreover, convergence occurs with a linear rate as $\varepsilon \rightarrow 0^+$ 
$$ \| u_{\varepsilon}(x) - u(x)\|_{L^{\infty}} \lesssim_{a(x), f} \varepsilon$$
and it is not difficult to see that this estimate is sharp.  It is often the case that local averages are smoother objects and we show that a similar phenomenon occurs here: local averages can improve convergence by an order of magnitude. This can also be motivated by a formal multiple scale expansion: one would expect the second term to have mean value 0 and to oscillate at scale $\sim \varepsilon$ suggesting that such a result may hold.

\subsection{The Result} Our result shows that local averages converge an entire order faster whenever the coefficient $a(\cdot)$ satisfies a suitable condition. We show that, among other things, $a(x) = a(1-x)$ implies that
$$ \left| \frac{1}{\varepsilon} \int_{x-\varepsilon/2}^{x+\varepsilon/2}{ u_{\varepsilon}(y) dy} - u(x) \right| \lesssim_{a, f}  \varepsilon^2.$$
The result is not generally true and there is an explicitly given linear function $\ell_{\varepsilon}$
that has to be corrected for first to reach the same order of convergence.

\begin{thm}  Let $u_{\varepsilon}$ and $u$ be defined as above. There exists an affine corrector $\ell_{\varepsilon}(x)$ given by
 \begin{align*}
\ell_{\varepsilon}(x) = \varepsilon  \left(\int_{0}^{1}{f(x) dx}\right) \left( \int_{0}^{1}{ a(y)^{-1}(y-\frac12)dy}\right) x +  \varepsilon\left(\int_{0}^{1}{ \int_{0}^{y}{f(z)dz} dy}\right)  \int_{-\frac12}^{\frac12}{ \int_{0}^{y}{a\left(z\right)^{-1}  dz} dy}
\end{align*}
such that for $\varepsilon^{-1} \in \mathbb{N}$ and $x \in (\varepsilon, 1-\varepsilon)$
$$ \left| \frac{1}{\varepsilon} \int_{x-\varepsilon/2}^{x+\varepsilon/2}{ u_{\varepsilon}(y) dy} + \ell_{\varepsilon}(x) - u(x) \right| \lesssim_{a, f}  \varepsilon^2.$$
\end{thm}

We note that the algebraic structure of $\ell_{\varepsilon}$ implies that $\ell \equiv 0$ whenver $a(x) = a(1-x)$, however, the actual condition is actually quite a bit more general than
that. We also note that $\ell$ vanishes for certain types of right-hand sides $f$ (a space of codimension 2 that contains, for example, the function $f(x) = x^2 - x +1/6$). This seems like an interesting phenomenon.\\

We are not aware of any result at that level of precision. However, there is a great deal of related work at a greater level of generality \cite{arm, arm2, ch, glorm1, glor0, glorr, glor1, glor2, koz, jikov, sm}.
The result is closely related to Allaire's two-scale approach \cite{all1, all2} and could be derived from a two-scale expansion. It remains to be seen whether our result has analogues in higher dimensions
or whether it is a one-dimensional 'miracle'. 
Our result deals with integrating an indicator function $\chi_{[-\varepsilon/2, \varepsilon/2]}$ against $u_{\varepsilon}$,
one could ask what happens if one were to take other averaging operators such as $\chi_{[-\varepsilon/2, \varepsilon/2]} * \chi_{[-\varepsilon/2, \varepsilon/2]}$ and whether there are higher-order analogues of our result in one dimension.

\subsection{Another proof.}
We conclude the paper with a description of an approach to homogenization that we came across by coincidence.
 A short summary is as follows:
\begin{enumerate}
\item Extend the elliptic equation in time by turning the elliptic operator $Lu = f$ into a parabolic operator $(\partial_t - L)u=f$; the solution becomes stationary in time.
\item Use the Feynman-Kac formula to produce reproducing identities for $u(x)$ and $u_{\varepsilon}(x)$ by writing them as weighted integral averages over
a neighborhood around $x$.
\item Algebraic manipulation leads to an inequality of the form
$$ |u_{\varepsilon}(x) - u(x)| \leq \mbox{error} + \left( \mbox{solution of heat equation starting from}~|u_{\varepsilon} - u|\right).$$
This type of bootstrapping can then be exploited to obtain $L^{\infty}-$estimates for $u_{\varepsilon} - u$.
\end{enumerate}

Generally, this seems to reduce the problem of obtaining quantitative estimates to controlling the diffusion induced by the infinitesimal operators. That underlying idea is
not new, we refer to work of Gloria \& Otto \cite{glor1, glor2} in the stochastic setting. 
We illustrate this technique for the most basic case discussed above in \S 3 and hope that it might have more general applications.

\section{Proof of the Theorem }
\begin{proof} We make use of the explicit solution formula
$$ u_{\varepsilon}(x) = \int_{0}^{x}{ a\left(\frac{y}{\varepsilon}\right)^{-1} \left( c_{\varepsilon} - \int_{0}^{y}{f(z) dz} \right) dy},$$
where
$$ c_{\varepsilon} = \left( \int_{0}^{1}{ a\left(x\right)^{-1} dx}\right)^{-1} \int_{0}^{1}{ \left(a\left(\frac{x}{\varepsilon}\right)^{-1} \int_{0}^{x}{f(y)dy}\right)dx}.$$
We will now use this formula to analyze
$$ \frac{1}{\varepsilon} \int_{x-\varepsilon/2}^{x+\varepsilon/2}{ u_{\varepsilon}(y) dy} =  \frac{1}{\varepsilon} \int_{x-\varepsilon/2}^{x+\varepsilon/2}{
\int_{0}^{y}{ a\left(\frac{z}{\varepsilon}\right)^{-1} \left( c_{\varepsilon} - \int_{0}^{z}{f(w) dw} \right) dy}
}.$$
The analysis decouples into three parts: analyzing the leading oscillation term, suitably approximating the constant $c_{\varepsilon}$ and analyzing the term involving the function $f$
and we will carry out the argument in that order.\\

\textit{The leading oscillation term.} The first term is the easiest since it grows linearly. 
Differentiation in $x$ leads to
$$ \frac{\partial}{\partial x}   \frac{1}{\varepsilon} \int_{x-\varepsilon/2}^{x+\varepsilon/2} \int_{0}^{y}{a\left(\frac{z}{\varepsilon}\right)^{-1}  dz} dy = \frac{1}{\varepsilon} \int_{x-\varepsilon/2}^{x+\varepsilon/2}{a_{}\left(\frac{y}{\varepsilon}\right)^{-1} dy} = \int_{0}^{1}{a(y)^{-1}dy}$$
which implies that the function is linear and
$$  \frac{1}{\varepsilon} \int_{x-\varepsilon/2}^{x+\varepsilon/2} \int_{0}^{y}{a\left(\frac{z}{\varepsilon}\right)^{-1}  dz} dy = \left(\int_{0}^{1}{a(y)^{-1}dy}\right) x + \frac{1}{\varepsilon}
 \int_{-\varepsilon/2}^{\varepsilon/2} \int_{0}^{y}{a\left(\frac{z}{\varepsilon}\right)^{-1}  dz} dy.$$

\textit{The constant $c_{\varepsilon}$.} We now compute the constant $c_{\varepsilon}$. This computation makes explicit use of $\varepsilon^{-1} \in \mathbb{N}$. We split the integral into the basic intervals
\begin{align*}
\int_{0}^{1}{ \left(a\left(\frac{x}{\varepsilon}\right)^{-1} \int_{0}^{x}{f(y)dy}\right)dx} = \sum_{k=0}^{\varepsilon^{-1}-1}{ \int_{k \varepsilon}^{(k+1)\varepsilon}{ \left(a\left(\frac{x}{\varepsilon}\right)^{-1} \int_{0}^{x}{f(y)dy}\right)dx}}
\end{align*}
and use 
\begin{align*}
 \int_{k \varepsilon}^{(k+1)\varepsilon}  a\left(\frac{x}{\varepsilon}\right)^{-1} \int_{0}^{x} f(y)dy dx 
&= \int_{k \varepsilon}^{(k+1)\varepsilon} a\left(\frac{x}{\varepsilon}\right)^{-1}\int_{0}^{(k+1/2)\varepsilon}f(y)dy dx\\
&+ \int_{k \varepsilon}^{(k+1)\varepsilon}   a\left(\frac{x}{\varepsilon}\right)^{-1} \int_{(k+1/2)\varepsilon}^{x}f(y)dydx.
\end{align*}
A Taylor expansion up to first order, using $f \in C^1$, shows that
\begin{align*}
  \int_{k \varepsilon}^{(k+1)\varepsilon}{  a\left(\frac{x}{\varepsilon}\right)^{-1}   \int_{(k+\frac12)\varepsilon}^{x}{f(y)dy}dx} &= \varepsilon^2 f((k+1/2)\varepsilon)\int_{0}^{1}{a(x)^{-1}(x-\frac12) dx} \\
&+ \mathcal{O}(\|f'\|_{L^{\infty}} \varepsilon^3)
\end{align*}
Altogether, this shows that
\begin{align*}
 \int_{0}^{1}  a\left(\frac{x}{\varepsilon}\right)^{-1} \int_{0}^{x} f(y)dy &= \sum_{k=0}^{\varepsilon^{-1}-1}  \int_{k \varepsilon}^{(k+1)\varepsilon}{   a\left(\frac{x}{\varepsilon}\right)^{-1}  \int_{0}^{(k+\frac12)\varepsilon}{f(y)dy} dx} \\
&+ \varepsilon^2 \left(\int_{0}^{1}{a(x)^{-1}(x-\frac12) dx}  \right)  \sum_{k=0}^{\varepsilon^{-1}-1}  f((k+\frac12)\varepsilon) \\
&+ \mathcal{O}(\|f'\|_{L^{\infty}} \varepsilon^2).
\end{align*}
The first term simplifies due to the periodicity of $a(\cdot)$ to
$$  \sum_{k=0}^{\varepsilon^{-1}-1}  \int_{k \varepsilon}^{(k+1)\varepsilon}{   a\left(\frac{x}{\varepsilon}\right)^{-1}  \int_{0}^{(k+1/2)\varepsilon}{f(y)dy} dx} =
 \varepsilon \left(\int_{0}^{1}{a(z)^{-1} dz} \right) \sum_{k=0}^{\varepsilon^{-1}-1}   \int_{0}^{(k+1/2)\varepsilon}{f(y)dy} $$
There are two sums that are left to be evaluated: one is essentially the midpoint rule, a Taylor expansion shows that
$$
  \int_{0}^{(k+1/2)\varepsilon}{f(y)dy} = \frac{1}{\varepsilon} \int_{k \varepsilon}^{(k+1)\varepsilon}{ \int_{0}^{y}{f(z)dz} dy} + \mathcal{O}(\|f'\|_{L^{\infty}} \varepsilon^2)
$$
In combination, we obtain for the first sum that
\begin{align*}
 \sum_{k=0}^{\varepsilon^{-1}-1}   \int_{0}^{(k+1/2)\varepsilon}{f(y)dy} &=  \sum_{k=0}^{\varepsilon^{-1}-1} \left(  \frac{1}{\varepsilon} \int_{k \varepsilon}^{(k+1)\varepsilon}{ \int_{0}^{y}{f(z)dz} dy} + \mathcal{O}(\|f'\|_{L^{\infty}} \varepsilon^2) \right) \\
&=  \frac{1}{\varepsilon} \int_{0}^{1}{ \int_{0}^{y}{f(z)dz} dy} + \mathcal{O}(\|f'\|_{L^{\infty}} \varepsilon).
\end{align*}
The second sum follows from another application of the midpoint rule in the form
$$ \varepsilon f((k+1/2)\varepsilon) =   \int_{k\varepsilon}^{(k+1)\varepsilon}{ f(x) dx} + \mathcal{O}(\|f''\|_{L^{\infty}} \varepsilon^3)$$
to simplify
\begin{align*}
 \varepsilon^2 \left(\int_{0}^{1}{a(x)^{-1}(x-1/2) dx}  \right)  \sum_{k=0}^{\varepsilon^{-1}-1}  f((k+1/2)\varepsilon) &= \varepsilon  \left(\int_{0}^{1}{a(x)^{-1}(x-1/2) dx}  \right) \int_{0}^{1}{f(x) dx} \\
&+ \mathcal{O}( \|f''\|_{L^{\infty}} \varepsilon^2)
\end{align*}
Collecting all these computations, we see that
\begin{align*} c_{\varepsilon} = \left(\int_{0}^{1}{ \int_{0}^{y}{f(z)dz}} dy\right)  &+ \left(\int_{0}^{1}{f(x) dx}\right) \left( \int_{0}^{1}{ a(x)^{-1}(x-1/2)dx}\right)    \\
&+ \mathcal{O}(\varepsilon^2 \|f'\|_{L^{\infty}})  + \mathcal{O}(\varepsilon^2 \|f''\|_{L^{\infty}}).
\end{align*}

\textit{The remaining term.} We now analyze the remaining term 
$-J$ given by
$$ J = \frac{1}{\varepsilon} \int_{x-\varepsilon/2}^{x+\varepsilon/2}{\int_{0}^{y}{ a\left(\frac{z}{\varepsilon}\right)^{-1}  \int_{0}^{z}{f(w) dw}  dy}}.$$
As before, we gain some insight into the expression by differentiating it first in $x$ and using yet another Taylor expansion (omitting higher order terms)
\begin{align*}
 \frac{\partial}{\partial x} J &= \frac{\partial}{\partial x} \frac{1}{\varepsilon} \int_{x-\varepsilon/2}^{x+\varepsilon/2}{\int_{0}^{y}{ a\left(\frac{z}{\varepsilon}\right)^{-1} \int_{0}^{z}{f(w) dw}  dz dy}} \\
&=\frac{1}{\varepsilon} {\int_{x-\varepsilon/2}^{x+\varepsilon/2}{ a\left(\frac{y}{\varepsilon}\right)^{-1} \int_{0}^{y}{f(z) dz} dy}} \\
&= \frac{1}{\varepsilon} \int_{x-\varepsilon/2}^{x+\varepsilon/2}{ a\left(\frac{y}{\varepsilon}\right)^{-1} \left( \int_{0}^{x}{f(z) dz} + (y-x)f(x) + \frac{(y-x)^2}{2} f'(x)  \right)dy}.
\end{align*}
The quadratic term is so small that it has no big effect
\begin{align*}
\left| \frac{1}{\varepsilon} \int_{x-\varepsilon/2}^{x+\varepsilon/2}{ a\left(\frac{y}{\varepsilon}\right)^{-1}  \frac{(y-x)^2}{2} f'(x) dy} \right| \lesssim_{a(\cdot)} \varepsilon^2 \|f'\|_{L^{\infty}}
\end{align*}
 and we obtain
\begin{align*}
 \frac{\partial}{\partial x} J  &= \frac{1}{\varepsilon} \int_{x-\varepsilon/2}^{x+\varepsilon/2}{ a\left(\frac{y}{\varepsilon}\right)^{-1} \left( \int_{0}^{x}{f(z) dz} + (y-x)f(x)  \right)dy} + \mathcal{O}( \varepsilon^2 \|f'\|_{L^{\infty}}) \\
&=  \left(  \int_{x-\varepsilon/2}^{x+\varepsilon/2}{ a\left(\frac{y}{\varepsilon}\right)^{-1} dy} \right)\left(\int_{0}^{x}{f(z) dz}\right) + \frac{f(x) }{\varepsilon} \int_{x-\varepsilon/2}^{x+\varepsilon/2}{ a\left(\frac{y}{\varepsilon}\right)^{-1}(y-x) dy} +  \mathcal{O}( \varepsilon^2 \|f'\|_{L^{\infty}}) \\
&= \left(  \int_{0}^{1}{ a\left(y\right)^{-1} dy} \right)\left(\int_{0}^{x}{f(z) dz}\right) + \frac{f(x) }{\varepsilon} \int_{x-\varepsilon/2}^{x+\varepsilon/2}{ a\left(\frac{y}{\varepsilon}\right)^{-1}(y-x) dy} +  \mathcal{O}( \varepsilon^2 \|f'\|_{L^{\infty}}).
\end{align*}

We study the error term by remarking that the function
$$ h(x) = \frac{1}{\varepsilon}\int_{x-\varepsilon/2}^{x+\varepsilon/2}{ a\left(\frac{y}{\varepsilon}\right)^{-1}(y-x) dy}$$
satisfies $\| h\|_{L^{\infty}} \lesssim \varepsilon$ (which is obvious) and
\begin{align*}
 \frac{1}{\varepsilon} \int_{x}^{x+\varepsilon}{h(y) dy} &=   \frac{1}{\varepsilon}\int_{x}^{x+\varepsilon}{\int_{y-\varepsilon/2}^{y+\varepsilon/2}{ a\left(\frac{z}{\varepsilon}\right)^{-1}(z-y) dz dy}} = 0.
\end{align*}
This implies, using another Taylor expansion of $f$, that
$$ \left| \int_{0}^{x}{ \frac{f(y) }{\varepsilon} \int_{y-\varepsilon/2}^{y+\varepsilon/2}{ a\left(\frac{z}{\varepsilon}\right)^{-1}(z-y) dz}dy dx} \right| \lesssim_{a(\cdot)} \varepsilon^2 \|f''\|_{L^{\infty}}.$$
This shows that
\begin{align*}
J &=  \left(  \int_{0}^{1}{ a\left(y\right)^{-1} dy} \right) \int_{0}^{x}{ \int_{0}^{y}{f(z) dz}dy} \\
 &+  \frac{1}{\varepsilon} \int_{-\varepsilon/2}^{+\varepsilon/2}{ a\left(\frac{x}{\varepsilon}\right)^{-1} \int_{0}^{x}{f(y) dy}  dx}+ \mathcal{O}_{a(\cdot)}(\|f'\|_{L^{\infty}}\varepsilon^2)
\end{align*}
\textit{Conclusion.} Collecting all the various estimates and terms, we now see that
\begin{align*}
 \frac{1}{\varepsilon} \int_{x-\varepsilon/2}^{x+\varepsilon/2}{ u_{\varepsilon}(y) dy} &=  \left(\int_{0}^{1}{a(y)^{-1}dy}\right) x   \left(\int_{0}^{1}{ \int_{0}^{y}{f(z)dz}} dy\right)  -   \left(  \int_{0}^{1}{ a\left(y\right)^{-1} dy} \right) \int_{0}^{x}{ \int_{0}^{y}{f(z) dz}dy}  \\
& + \mathcal{O}_{a(\cdot)}( (\|f'\|_{L^{\infty}} + |f''\|_{L^{\infty}} ) \varepsilon^2) + \ell(x)
\end{align*}
where
\begin{align*}
 \ell(x) &= \varepsilon  \left(\int_{0}^{1}{f(x) dx}\right) \left( \int_{0}^{1}{ a(y)^{-1}(y-1/2)dy}\right) x \\
&+  \varepsilon\left(\int_{0}^{1}{ \int_{0}^{y}{f(z)dz} dy}\right)    \frac{1}{\varepsilon}  \int_{-\varepsilon/2}^{\varepsilon/2}{ \int_{0}^{y}{a\left(\frac{z}{\varepsilon}\right)^{-1}  dz} dy}.
\end{align*}
The main term is exactly the solution formula for the homogenized equation and from this the desired result follows.
\end{proof}

\section{Outline of another approach}
In this section we outline another approach to homogenization and illustrate it in its most basic form for the problem 
\begin{align*}
-\left(a\left(\frac{x}{\varepsilon}\right) u_{\varepsilon}'(x)\right)' &= f(x) \quad \mbox{for} ~x \in (0,1)\\
u_{\varepsilon}(0) = 0& = u_{\varepsilon}(1)
\end{align*}
where $a:\mathbb{R} \rightarrow \mathbb{R}_{+}$ is a strictly positive function with period 1, the Dirichlet
boundary conditions $u_{\varepsilon}(0) = 0 = u_{\varepsilon}(1)$ are prescribed, and $f$ is assumed to be smooth.

\subsection{Two reproducing identities.}
The key idea behind our approach is to extend the equation in time by making it parabolic; since $u_{\varepsilon}$ actually solves the equation, it becomes a stationary-in-time solution
of a heat equation which can be studied with probabilistic methods (these ingredients are, of course, classical for the study of homogenization of parabolic equations, see \cite[Section 2]{jikov}). We derive reproducing identities for both $u_{\varepsilon}$ and $u$ and will use
them to bootstrap a bound. More precisely, we will study solutions of the heat equation
\begin{align*}
\frac{\partial}{\partial t} u_{\varepsilon}(t,x)  -\left(a^{}\left(\frac{x}{\varepsilon}\right) u_{\varepsilon}'(t,x)\right)' &= f(x) \quad \mbox{on} ~(0,1)\\
u_{\varepsilon}(t,0) = 0&, u_{\varepsilon}(t,1) = 0.
\end{align*}
By construction, $u_{\varepsilon}(t,x) = u_{\varepsilon}(x)$ is a stationary solution in time. However, this heat equation also has a probabilistic interpretation. 
We use $\omega_x(t)$ to denote Brownian motion started in $x$ and running for $t$ units of time subjected to diffusivity $a^{}(x/\varepsilon)$. If it exits the domain $[0,1]$ before $t$ units of time
have passed, we assume that $\omega_x(t)$ remains stationary at the point of exit (the boundary is 'sticky'). Then the Feynman-Kac formula implies
$$ u_{\varepsilon}(x) = \mathbb{E} u_{\varepsilon}(\omega_x(t)) +  \mathbb{E}\int_{0}^{t}{ f(\omega_x(s)) ds}.$$
We introduce a second Brownian motion $\nu_x$ which is merely governed by diffusion w.r.t. to the homogenized diffusion coefficient (and 'stickiness' w.r.t. the boundary) which allows
us to write the homogenized equation in a similar manner
$$ u_{}(x) = \mathbb{E} u_{}(\nu_x(t)) +  \mathbb{E}\int_{0}^{t}{ f(\nu_x(s)) ds}.$$
We will show that these two reproducing identities in combination with basic bounds on the two types of diffusions are sufficient to bootstrap everything into a quantitative estimate.

\subsection{Bootstrapping}
It will be more convenient to work with probability distributions. We will denote the distribution of $\omega_x(t)$ in point $y$ by $k_{t}(x,y)$ and the distribution of
$\nu_x(t)$ in point $y$ by $\ell_t(x,y)$. Note that both are continuous in the interval $(0,1)$ but do have atomic masses at the boundary points of the interval because these endpoints
are sticky. In particular, they are not probability distributions on $(0,1)$ because they do not integrate to 1 in that interval.
Since we fixed boundary conditions to be 0, this allows us to write
$$u_{\varepsilon}(x) = \int_{0}^{1}{ k_t(x,y) u_{\varepsilon}(y) dy} + \int_{0}^{t}{ \int_{0}^{1}{ k_s(x,y) f(y) dy} ds}$$
as well as
$$u_{}(x) = \int_{0}^{1}{ \ell_t(x,y) u_{\varepsilon}(y) dy} + \int_{0}^{t}{ \int_{0}^{1}{ \ell_s(x,y) f(y) dy} ds}.$$
Their difference can be written as
\begin{align*}
u_{\varepsilon}(x) - u_{}(x) &=  \int_{0}^{t}{ \int_{0}^{1}{ (k_s(x,y)- \ell_s(x,y)) f(y) dy} ds} \\ 
&+\int_{0}^{1}{ (k_t(x,y) - \ell_t(x,y)) u_{\varepsilon}(y) dy}\\
&+ \int_{0}^{1}{ \ell_t(x,y) (u_{\varepsilon}(y) - u(y)) dy}.
\end{align*}
We will introduce the first two terms as error estimates 
$$ \delta =  \max_{0 \leq x \leq 1}  \left|  \int_{0}^{t}{ \int_{0}^{1}{ (k_s(x,y)- \ell_s(x,y)) f(y) dy} ds}  \right| + \left|  \int_{0}^{1}{ (k_t(x,y) - \ell_t(x,y)) u_{\varepsilon}(y) dy} \right| $$
and introducing $\phi(x) = |u_{\varepsilon}(x) - u(x)|$, this allows us to estimate
\begin{align*}
 \phi(x) \leq \delta +  \int_{0}^{1}{ \ell_t(x,y) \phi(y) dy}
\end{align*}
The interesting twist comes from interpreting the third term as the solution of yet another heat equation. Indeed, we have that
$$ w_t(x) =   \int_{0}^{1}{ \ell_t(x,y) \phi(y) dy}$$
solves the heat equation
\begin{align*}
\frac{\partial}{\partial t} w_t(x) &= \left( \overline{a} w_{t}'(x)\right)' \quad \mbox{for} ~x \in (0,1)\\
w_{t}(0) &= 0 = w_{t}(1) \quad \mbox{for all}~t\\
w_0(x) &= \phi(x)
\end{align*}
with $\overline{a}$ being the homogenized coefficient associated to the variable coefficient $a(\cdot)$. Using $e^{t \Delta_a}$ to denote heat propagator (with Dirichlet boundary conditions) associated to the problem, we can write our inequality as
$$ \phi(x) \leq \delta + \left(e^{t \Delta_a} \phi\right)(x).$$
This leads to a bound on the maximum size of $\phi$: if $\phi$ was very, very large, then a short application of the heat equation would diminish it quite a bit (this clearly requires Dirichlet conditions) and the inequality would fail. We now make this intuition precise.
\begin{lemma} If $\phi:[0,1] \rightarrow \mathbb{R}_{}$ satisfies $\phi(0) = 0 = \phi(1) $ and, for some $0 < t \leq 1$,
$$ \phi(x) \leq \delta + \left(e^{t \Delta_a} \phi\right)(x),$$
then
$$ \max_{0 \leq x \leq 1}{\phi(x)} \lesssim_{a} \frac{\delta}{t}.$$
\end{lemma}
\begin{proof} We use monotonicity of the heat equation to iterate the argument. More precisely, using the assumption twice yields
$$ \phi(x) \leq \delta + \left(e^{t \Delta_a} \phi\right)(x) \leq \phi(x) \leq \delta + \left(e^{t \Delta_a}  \delta + \left(e^{t \Delta_a} \phi\right)\right)(x).$$
However, the heat flow of a constant can be bounded from above by the constant and this, together with the semigroup property, implies
$$ \phi(x) \leq  2 \delta + \left(e^{2 t \Delta_a} \phi\right)(x).$$
Iterating the argument shows that, for every integer $k \in \mathbb{N}$
$$ \phi(x) \leq  k \delta + \left(e^{k t \Delta_a} \phi\right)(x).$$
The remaining ingredient is the following fact: there exists $c_{\overline{a}} > 0$, depending only on $\overline{a}$, such that for all bounded functions $f: [0,1] \rightarrow \mathbb{R}$ and all $1 \leq t \leq 2$
$$ \max_{0 \leq x \leq 1}{e^{t \Delta_a}f(x)} \leq (1-c_{\overline{a}}) \max_{0 \leq x \leq 1}{f(x)}.$$
Once this statement is known, we can estimate, for $k \sim t^{-1}$,
\begin{align*}
\max_{0 \leq x \leq 1} \phi(x) &\leq  \max_{0 \leq x \leq 1}  k \delta + \left(e^{k t \Delta_a} \phi\right)(x) \\
&=  k \delta +  \max_{0 \leq x \leq 1}   \left(e^{k t \Delta_a} \phi\right)(x) \\
&\leq  k \delta +  (1-c_{\overline{a}})\max_{0 \leq x \leq 1}    \phi(x)
\end{align*}
which then implies the desired result. It remains to establish the helpful fact which, in turn, follows
from the fact that the exit probability of Brownian motion out of the unit interval $[0,1]$ at time $t \sim 1$ is comparable to $\sim 1$ uniformly on the entire interval.
\end{proof}

\subsection{Estimating the error terms.} It remains to estimate the size of the error terms. Both of these require a good understanding of the behavior of
$k_t(x,y)$. The main ingredient is as follows (see Fig 1.): the average value of $k_t(x,\cdot)$ over intervals of length being a multiple of $\varepsilon$ away
from $x$ and of length $\varepsilon$ coincides \textit{exactly} with that of the homogenized problem inducing $\ell_t(x,\cdot)$.

\begin{lemma} For all $t>0$ and any $k \in \mathbb{Z}$ such that $[x+k \varepsilon,x+ (k+1) \varepsilon] \subset (0,1)$
$$ \int_{x + k\varepsilon}^{x+ (k+1)\varepsilon}{ k_t(x,y) dy} = \int_{x + k\varepsilon}^{x+ (k+1)\varepsilon}{ \ell_t(x,y) dy}.$$
\end{lemma}
\begin{proof} It suffices to understand diffusion induced by $a(x/\varepsilon)$ with diffusion induced by $\overline{a}$. This turns out to be rather simple: instead
of interpreting the homogenization problem as variable-strength diffusion governed by $a(x/\varepsilon)$, we may interpret it as classical constant-coefficient
diffusion on a one-dimensional manifold whose metric is determined by $a(x/\varepsilon)$. 
\end{proof}

We will now use Lemma 2 to estimate the size of the error term
$$ \delta =   \left|  \int_{0}^{t}{ \int_{0}^{1}{ (k_s(x,y)- \ell_s(x,y)) f(y) dy} ds}  \right| + \left|  \int_{0}^{1}{ (k_t(x,y) - \ell_t(x,y)) u_{\varepsilon}(y) dy} \right|.$$
Since the total weight assigned by $k_s$ and $\ell_s$ to any interval of length $\varepsilon$ coincides, we may argue that, for all $s > 0$,
$$  \left| \int_{0}^{1}{ (k_s(x,y)- \ell_s(x,y)) f(y) dy} \right| \leq \varepsilon\|f\|_{L^{\infty}} + \varepsilon \|f'\|_{L^{\infty}}$$
as well as
$$  \left|  \int_{0}^{1}{ (k_t(x,y) - \ell_t(x,y)) u_{\varepsilon}(y) dy} \right| \leq \varepsilon \|u_{}'\|_{L^{\infty}} + \varepsilon \|u_{\varepsilon}'\|_{L^{\infty}}.$$
Altogether, this yields the estimate
$$\|u_{\varepsilon} - u\|_{L^{\infty}} \lesssim  \varepsilon\|f\|_{L^{\infty}} + \varepsilon \|f'\|_{L^{\infty}}+ \frac{ \varepsilon \|u_{}'\|_{L^{\infty}} + \varepsilon \|u_{\varepsilon}'\|_{L^{\infty}}}{t}.$$
Setting $t=1$ reduce the convergence rate to a priori estimates on $u$ and $u_{\varepsilon}$.\\

\textbf{Acknowledgment.} The author is indebted to Jessica Lin for a series of helpful discussions.

\end{document}